\documentclass{baustms}
\citesort

\usepackage{graphicx}
\theoremstyle{cupthm}
\newtheorem{thm}{Theorem}[section]

\newtheorem{cor}[thm]{Corollary}

\theoremstyle{cupdefn}
\newtheorem{defn}[thm]{Definition}
\theoremstyle{cuprem}

\numberwithin{equation}{section}
\newtheorem{example}[thm]{Example}

\begin{document}
\runningtitle{Position Vectors of Curves With Recpect to Darboux Frame}
\title{Position Vectors of Curves With Recpect to Darboux Frame in the Galilean Space $G^3$}
%% If there is more than one author, put \cauthor immediately before
%% the corresponding author.
%\cauthor %% mark the next author as corresponding author
\author[1]{Tevf{\.i}k \c{S}ah{\.i}n}
\address[1]{Department of Mathematics, Amasya University, Amasya, Turkey\email{tevfik.sahin@amasya.edu.tr, tevfiksah@gmail.com}}
\author[2]{Buket Ceylan D{\.i}r{\.i}\c{s}en}
\address[2]{Department of Mathematics, Amasya University, Amasya, Turkey\email{bceylandirisen@gmail.com}}

%% If there are several authors, list them here

%% List the authors, initials and surnames only, for the
%% running head (left hand page)
\authorheadline{ T. \c{S}ah{\.i}n and B. Ceylan D{\.i}r{\.i}\c{s}en}

%% If there is a dedication, include it here
%\dedication{Dedicated to ...}

\support{FMB-BAP 16-0213}

\begin{abstract}
In this paper, we investigate the position vector of a curve on the surface in the Galilean 3-space $G^3$. Firstly, the position vector of a curve with respect to the Darboux frame is determined. Secondly, we obtain the standard representation of the position vector of the curve with respect to Darboux frame in terms of the geodesic, normal curvature and geodesic torsion. As a result of this, we define the position vectors of geodesic, asymptotic and normal line along with some special curves with respect to Darboux frame. Finally, we elaborate on some examples and provide their graphs. 
\end{abstract}

%% - subject classification and keywords
%% 2010 American Mathematical Society Subject Classification
%% Provide only ONE primary classification
%% Four or five keywords or phrases

\maketitle

\section{Introduction}
The fundamental theorem of curves state that curves are determined by curvatures \cite{krey}. Thus, curvature functions provide us with some special and important information about curves.
For example, a circular helix is a geometric curve with curvature  $\kappa\equiv\textit{cons.}\neq0$, torsion $\tau\equiv\textit{cons.}\neq0$ \cite{mcc}. Straight lines and circles are curves that $\kappa\equiv 0$ and $\kappa\equiv cons.$, $\tau\equiv0$, respectively. Also, These curves are degenerate helices. Helices appear in many different branch of science such as engineering, biology, chemistry, CAD, etc. 

In addition, curvature functions gives us information about not only curves but also surfaces on which curves lie. The curvatures $\kappa_g \,\, \textit{geodesic curvature}, \kappa_n \,\, \textit{normal curvature}$ and $\tau_g $ \textit{geodesic torsion} charactarize geodesic, asymptotic curve, and line of curvature, respectively.
The curves emerges from the solution of some important physical problems. They are also important in the theory of curves and surfaces.
For example, geodesics arise from the problem of finding 'shortest curves' joining two points of a surface $M$. It was first considered by Johann Bernoulli (1697).
 Obviously this is a typically problem of calculus of variations. 
% Johann Bernoulli had used the term  "shortest curves" (linea brevissima) for the solve of the above problem. In 1844 the term "geodesic  curve" replaced "shortest curve" by J. Liouvile (1809-82) \cite{mcc}. Thus they satisfy the condition of Archimedes for being lines.
%The problem becomes very simple if $M$ is a plane. In this case shortest curves in plane joining $P$ and $Q$  is the segment of a straight line which has the endpoints $P$ and $Q$. 
%On an arbitrary surface the curves of minimal length will play a role similar to that of straight lines in a plane \cite{krey}. 
Also, a curve  $C$ on a surface $M$ is called a \textit{geodesic curve} or \textit{geodesic} if its geodesic curvature $\kappa_g$ vanishes identically \cite{krey}.
In what follows, we state three different definition lines in planes. We want to emphasize that geodesics can be seen as extension of this idea to curves in surfaces.

A "line" on a surfaces can be seen as extension of the familiar properties of lines in the plane: For example, lines are 
(1) The curves of shortest length joining two points (Archimedes). 
(2) The curves of plane curvature identically zero (Huygens, Leibniz, Newton). 
(3) The curves whose tangent and its derivative are linearly dependent \cite{mcc}.

%
%Straight lines can be characterized by the property that their curvature vanishes identically (as Huygens, Leibniz, Newton). 
%
%The geodesic curvature is the generalizition to an arbitrary surface of the signed curvature of a plane curve \cite{krey}.
%
As stated in \cite{aligal}, the problem of the determination of the position 
vector of a space curve with respect to the Frenet frame is still open in the Euclidean space. Generally, it is hard to solve this problem. However, it is solved for some special curves such as plane curves, helix and slant helix \cite{alimin,alieuc,izu}. On the other hand, in the Galilean space $G^3$, the foregoing problem is solved for all curves \cite{aligal}.
%solved in some special cases such as: the case of a plane curve, a generalized helix and a slant helix \cite{alimin,aliturmin,alilopmin,alidet,alieuc,izu}.  This problem is neither solved in other cases of the space curve in Euclidean 3-space nor Minkowski 3-space. However, this problem is solved for all curves in the Galilean 3-space  $G_3$  \cite{aligal} . Ali \cite{aligal} (bu \c cal\i \c smada ???) determine the position vector of an arbitrary curve with respect to the Frenet frame. Also, in this study \cite{aligal} is given the position vector of some special curves in $G_3$ such as: plane curve, helix, general helix, Salkowski and anti-Salkowski curves. 
The main aim of this study is to solve the above problem for all curves on a surface in $G^3$ with respect to the Darboux Frame. Firstly, we determine the position vector of a curve on a surface in $G^3$ in terms of geodesic, normal curvature and geodesic torsion with respect to the Darboux and standard frame. Secondly, we shall give position vectors of some special curves such as geodesic, asymptotic curve, line of curvature on a surface in $G^3$.

Also, we will relate foregoing curves with helix, Salkowski curve and anti-Salkowski curve (see (\ref{helsal}) ).
That is, we shall give special cases of these curves  such as: geodesics that are circular helix, genaralized helix or Salkowski, etc. Furthermore, we provide graphs of some special curves.  

Last but not least, we want to emphasize that the results of this study can be extended to families of surfaces that have common geodesic curve.
%case. However, this problem is solved in some special cases such as: the case of a 
%plane curve $(\tau = 0)$, the case of a helix ($\kappa$ and $\tau$ are both non-vanishing constant). 
%Recently, Ali [4, 5] adapted fundamental existence and uniqueness theorem for 
%space curves in Euclidean space $\mathbb{E}^3$ and constructed a vector differential equation to solve this problem in the case of a general helix ( $\frac{\tau}{\kappa}$ is constant) and in the case of a slant helix ( $\frac{\kappa^2}{(\kappa^2+\tau^2)^{3/2}}\Big(\frac{\tau}{\kappa}\Big)'$ is constant). 

%%%%%%%%%%%%%%%%%%%%%%%%%%%%
%%%%%%%%%%%%%%%%%%%%%%%%%%%%
%%%%%%%%%%%%%%%%%%%%%%%%%%%%
%%%%%%%%%%%%%%%%%%%%%%%%%%%%

\section{Preliminaries}
%\section{Preliminaries on the Galilean Geometry}

As it is well known, Galilean geometry  is associated with the Galilean principle of relativity. 
The Galilean space $G^{3}$ is one of the Cayley-Klein spaces equipped with the projective metric of signature $\left( 0,0,+,+\right) $ \cite{mol}. The absolute figure of the Galilean space is the ordered triple $\{w,f,I\}$, where $w$ is an ideal (absolute) plane, %in the real three-dimensional projective space $P^3(\mathbb{R} )$
$f$  is a line (absolute line) in $w$, and $I$ is a fixed eliptic involution of points of $f$. 

%Let $P$ be any point of $\mathbb{R}$ with affine coordinates $(x, y, z)$. We write $(x, y, z)$ in the form $(\frac{x_1}{x_0}, \frac{x_2}{x_0}, \frac{x_3}{x_0})$ where $x_0$ is a common denominator. We say that $(x_0, x_1, x_2, x_3)$ are \textit{homogeneous coordinates of $P$}. Thus, the homogeneous coordinates $(x_0, x_1, x_2, x_3)$ and $\rho(x_0, x_1, x_2, x_3)$ refer to the same point for all $\rho \in \mathbb{R}-\{0\}$ \cite{tsuk}. We can now introduce homogeneous coordinates in $G_1^3$ in such a way that the absolute plane $w$ is given by $x_0=0$, the absolute line $f$ by $x_0=x_1=0$, and the hyperbolic involution $I$ by 
%$$(0: 0: x_2: x_3)\rightarrow(0: 0: x_3: x_2).$$
% The last condition is equivalent to the requirement that the conic $x_2^2-x_3^2=0$ is the absolute conic. Metric relations are introduced with respect to the absolute figure. In affine coordinates, the distance between points $P_i=(x_i, y_i, z_i), \, i=1,2$ is defined by: 
%$$ d(P_1, P_2)= \left\{
%     \begin{array}{lr}
%       \vert {x_2-x_1}\vert, &  \text{if} x_{1}\neq x_2,   \\
%     \\
%       \sqrt{\big\vert (y_2-y_1)^2-(z_2-z_1)^2\big\vert } ,& \text{if } x_{1}\neq x_2\,.
%       
%     \end{array}\right.$$

In non-homogeneous coordinates the group of motion of $G^3$ (i.e. the group of isometries of $G^{3}$) has the form
define :%
\begin{eqnarray}
\overline{x} &=&a_1+x,  \notag \\
\overline{y} &=&a_2+a_3x+y\cos \varphi +z\sin \varphi , \\
\overline{z} &=&a_4+a_5x-y\sin \varphi +z\cos \varphi,  \notag
\end{eqnarray}%
where $a_1, a_2, a_3, a_4, a_5$, and $\varphi$ are real numbers \cite{pav}. %Note that the group of motions of $G_{1}^{3}$ is a six-parameter group. It leaves invariant the absolute figure as well as the pseudo-Galilean distance of points.
If the first component of a vector is not zero, then the vector is called as non-isotropic, otherwise it is called isotropic vector \cite{pav}.
%A vector $V(v_1, v_2, v_3)$ in the $G_1^3$ is said to be non-isotropic if $v_1\neq 0$. Thus,
%all unit non-isotropic vectors are of the form $\left( 1,v_2,v_3\right) $. For
%isotropic vectors $x=0$ holds and four types of isotropic vectors: spacelike 
%$(v_2^{2}-v_3^{2}>0)$, timelike $(v_2^{2}-v_3^{2}<0)$ and two types of lightlike
%vectors $(v_2=\pm v_3)$  

The scalar product of two vectors $\mathbf{v}=(v_{1},v_{2},v_{3})$ and $\mathbf{w}=(w_{1},w_{2},w_{3})$ in $G^{3}$ is defined by
$$\mathbf{v}\cdot _{G}\mathbf{w} = \left\{
\begin{array}{lr}
v_{1}w_{1} , &  \text{if } v_{1}\neq 0 \text{ or } w_{1}\neq 0\, \ \  \ \\
v_{2}w_{2}+v_{3}w_{3} ,&  \text{if } v_{1}=0 \text{ and } w_{1}=0\,.
\end{array}\right.$$
If $\mathbf{v}\cdot _{G}\mathbf{w}=0$, then $\mathbf{v}$ and $\mathbf{w}$
are perpendicular. In particular, every isotropic vector is perpendicular to every non-isotropic vector.
The norm of $\mathbf{v}$ is defined by
$$\Vert \mathbf{v}\Vert_{G}=\sqrt{\vert\mathbf{v}\cdot_{G}\mathbf{v}\vert}.$$
Let $I\subset \mathbb R$ and let $\alpha :I\rightarrow G^{3}$ be a curve parameterized by arc length (we abbreviate as p.b.a.l) with curvature $\kappa>0$ and torsion $\tau$.
% then three vector fields $T, N$ and $B$ on $\alpha$ are unit vector fields that are mutually orthogonal at each point. The vector fields $T, N, B$  are called the \textit{Frenet frame field} on $\alpha$.
If $\alpha$ is a curve p.b.a.l. that is,
\begin{equation*}
\alpha \left( x\right) =\left( x,y\left( x\right) ,z\left( x\right) \right) ,
\end{equation*}%
then the Frenet frame fields are given by
\begin{eqnarray}
T\left(x\right) &=&\alpha ^{\prime }\left( x\right), 
%=\left( 1,y^{\prime}\left( x\right) ,z^{\prime }\left( x\right) \right)
\notag \\
N\left( x\right) &=& \frac{\alpha''(x)}{\Vert \alpha''(x)\Vert_{G}}
%\frac{1}{\kappa \left( x\right) }\left( 0,y^{\prime \prime }\left( x\right) ,z^{\prime \prime }\left( x\right) \right) ,
\\
B\left( x\right) &=&T(x)\times _{G}B(x) \\&=&\frac{1}{\kappa \left( x\right) }\left( 0,
-z^{\prime \prime }\left( x\right) , y^{\prime \prime }\left(
x\right) \right) ,  \notag
\end{eqnarray}%
where $\kappa \left( x\right) $ and $\tau \left(
x\right) $ are defined by%
\begin{equation}
\kappa \left( x\right) ={\Vert \alpha''(x) \Vert }_{G}, { \ \ }\tau
\left( x\right) =\frac{\det \left( \alpha ^{\prime }\left( x\right) ,\alpha
	^{\prime \prime }\left( x\right) ,\alpha ^{\prime \prime \prime }\left(
	x\right) \right) }{\kappa ^{2}\left( x\right) }\,.
\end{equation}%

%Also, where $\varepsilon =\pm 1$ determined by the criterion $\det (T,N,B)=\pm1$. 
Also, where $\times _{G}$ is the pseudo-Galilean cross product  defined by
\begin{equation}
\mathbf v\times _{G}\mathbf w=%
\begin{vmatrix}
0 & \mathbf e_{2} &\mathbf {e_{3}} \\ 
v_{1} & v_{2} & v_{3} \\ 
w_{1} & w_{2} & w_{3}%
\end{vmatrix}%
\end{equation}%
\newline
for $\mathbf{v=}\left( v_{1},v_{2},v_{3}\right) $ and $\mathbf{w=}\left(
w_{1},w_{2},w_{3}\right) $ \cite{pav1}. 
%This means that
%\begin{equation}
%\left\vert y^{\prime \prime }\left( x\right) ^{2}-z^{\prime \prime }\left(
%x\right) ^{2}\right\vert =\varepsilon \left( y^{\prime \prime }\left(
%x\right) ^{2}-z^{\prime \prime }\left( x\right) ^{2}\right) .
%\end{equation}%
%The curve $\alpha $ is time-like (resp. space-like) if $N(x) $
%is a space-like (resp. time-like) vector. From equation (2.4), we can infer that the principal normal vector is space-like if $\varepsilon =+1$ and time-like if $
%\varepsilon =-1.$ 
The vectors $T, N $ and $B$ are called the vectors of the tangent, the principal normal
and the binormal vector field, respectively \cite{pav1}. Therefore, the Frenet-Serret formulae can be written as
\begin{equation}
\begin{bmatrix}
T \\ 
N \\ 
B%
\end{bmatrix}%
^{\prime }=%
\begin{bmatrix}
0 & \kappa & 0 \\ 
0 & 0 & \tau \\ 
0 &- \tau & 0%
\end{bmatrix}%
\begin{bmatrix}
T \\ 
N \\ 
B%
\end{bmatrix}\,.
\end{equation}%

Frame fields constitute a very useful tool for studying curves and surfaces. However, the Frenet frame ${T, N, B}$ of $\alpha$ is not useful to describe the geometry of surface $M$. Since $N$ and $B$ in general will be neither tangent nor perpendicular to M. Therefore, we require another frame of $\alpha$ for study the relation between the geometry of $\alpha$ and $M$. There is such a frame field that is called Darboux frame field of $\alpha$ with respect to $M$. The Darboux frame field consists of the triple of vector fields ${T, Q, n}$. The first and last vector fields of this frame $T$ and $n$ are a unit tangent vector field of $\alpha$ and unit normal vector field of $M$ at the point $\alpha(x)$ of $\alpha$. Let $Q=n\times_{G}T$ be the tangential-normal. 

% For alignments use AmS-LaTeX constructions not \eqnarray.

%% - theorems and proofs
\begin{thm}Let  $\alpha :I\subset \mathbb{R}\rightarrow M\subset G^{3}$ be a unit-speed curve, and let {T, Q, n} be the Darboux frame field of $\alpha$ with respect to M. Then
	\begin{equation}\label{Darboux}
	\begin{bmatrix}
	T \\ 
	Q \\ 
	n
	\end{bmatrix}
	^{\prime }=
	\begin{bmatrix}
	0 & \kappa_g & \kappa_n \\ 
	0 & 0 & \tau_g \\ 
	0 & -\tau_g & 0
	\end{bmatrix}
	\begin{bmatrix}
	T \\ 
	Q \\ 
	n
	\end{bmatrix}\,.
	\end{equation}
	where $\kappa_g$ and $\kappa_n$ give the tangential and normal component of the curvature vector, and these functions are called the geodesic and the normal curvature, respectively \cite{sahin}.
% The optional material will be typeset as part of the theorem heading
\end{thm}

\begin{proof}	We have 
	\begin{equation}\label{c}
	\begin{split}
	T'&= (T'\cdot_{G}Q)Q+(T'\cdot_{G}n)n
	\\&= (\alpha''\cdot_{G}Q)Q+(\alpha''\cdot_{G}n)n
	\\&=\kappa_g Q + \kappa_n n.
	\end{split}
	\end{equation}
	The other formulae are proved in a similar fashion.
% the proof
\end{proof}

Also, (2.7) implies the important relations
\begin{equation}\label{kt}
\kappa^2(x)=\kappa^2_g(x)+\kappa^2_n(x),  \hskip .5cm \tau(x)=-\tau_g(x)+\frac{\kappa'_g(x)\kappa_n(x)-\kappa_g(x)\kappa'_n(x)}{\kappa^2_g(x)+\kappa^2_n(x)}
\end{equation} 
where $\kappa^2(x)$ and $\tau(x)$ are the square curvature and the torsion of $\alpha$, respectively.
We refer to \cite{pav, pav1, ros, yag} for detailed treatment of Galilean and pseudo-Galilean geometry.

%%%%%%%%%%%%%%%%%%%%%%%%%%%%%%%%%%%%%%%%%%%
%%%%%%%%%%%%%%%%%%%%%%%%%%%%%%%%%%%%%%%%%
%%%%%%%%%%%%%%%%%%%%%%%%%%%%%%%%%%%%%%%%%%
%%%%%%%%%%%%%%%%%%%%%%%%%%%%%%%%%%%%%%%%%%

\section{Position vectors of a curve in Galilean space}

In this section, we will get an arbitrary curve on a surface in $G^3$. We will analyze position vector of the curve with respect to the Darboux and standard frame in $G^3$.
\begin{thm}\label{teo1}
	The position vector $\beta(x)$ of an arbitrary curve on a surface with respect to the Darboux frame in the Galilean space $G^3$ is given by:
	\begin{equation}\label{eq00}
	\begin{split}
	\beta(x)=(x+c_1)\mathbf{T}&+\Bigg\{-\frac{(x+c_1)\kappa_n(x)}{\tau_g(x)}+\bigg(c_2-\int{f(x)\tau_g(x)\sin{[t(x)]}}\,dx\bigg)\sin{[t(x)]}\\&-\bigg(c_3+\int{f(x)\tau_g(x)\cos{[t(x)]}}\,dx\bigg)\cos{[t(x)]}\Bigg\}\mathbf{Q}\\&+\Bigg\{\bigg(c_2-\int{f(x)\tau_g(x)\sin{[t(x)]}}\,dx\bigg)\cos{[t(x)]}\\&+\bigg(c_3+\int{f(x)\tau_g(x)\cos{[t(x)]}}\,dx\bigg)\sin{[t(x)]}\Bigg\}\mathbf{n}
	\end{split}
	\end{equation}
	where $f(x)=\frac{\lambda_1(x)\kappa_g(x)}{\tau_g(x)}-\bigg(\frac{\lambda_1(x)\kappa_n(x)}{\tau_g(x)}\bigg)'\frac{1}{\tau_g(x)}$ and $t(x)=\int{\tau_g(x)}\,dx$.
\end{thm}

\begin{proof}Let $\beta(x)$ be an arbitrary curve on a surface in the $G^3$, then, we may express its position vectors with respect to the Darboux frame as follows:
	\begin{equation}\label{eq0}
	\beta(x)=\lambda_1(x)\mathbf{T}+\lambda_2(x)\mathbf{Q}+\lambda_3(x)\mathbf{n}
	\end{equation}
	where $\lambda_1(x), \lambda_2(x)$ and $\lambda_3(x)$ are differentiable functions of $x\in{I}\subset{\mathbb{R}}$. By differentiating (\ref{eq0}) and using (\ref{Darboux}), we get
	\begin{equation}\label{eq1}
	\begin{array}{rl}
	\lambda'_1(x)-1 &=0 \\
	\lambda_1(x)\kappa_g(x)+\lambda'_2(x)-\lambda_3(x)\tau_g(x) & =0\\
	\lambda_1(x)\kappa_n(x)+\lambda_2(x)\tau_g(x)+\lambda'_3(x) &=0
	
	\end{array}.
	\end{equation}
	The first equation of (\ref{eq1}) leads to
	\begin{equation}\label{eq1.k}
	\lambda_1(x)=x+c_1
	\end{equation}
	where $c_1$ is an arbitrary real constant.
	To solve (\ref{eq1}) for $\lambda_i$, we use the following change of variable $t =\int{\tau_g(x)dx}$ so that
	
	\begin{equation}\label{eq2.k}
	\begin{array}{rl}
	\lambda_1(t) &= (\lambda_1 \circ x)(t),\\
	 \tau_g(t) &= (\tau_g \circ x)(t),\\
	 \\
	\lambda_2(t)&=-\frac{\lambda_1(t)\kappa_n(t)}{\tau_g(t)}-\dot{\lambda_3}(t).
	\end{array}
	\end{equation}
	Here, "\,$\dot{}$\," stands for derivative with respect to $t$.
	
	Substituting (\ref{eq2.k}) into (\ref{eq1}) we get the following equation
	\begin{equation}
	\ddot{\lambda_3}(t)+\lambda_3(t)=\frac{\lambda_1(t)\kappa_g(t)}{\tau_g(t)}-\bigg(\frac{\lambda_1(t)\kappa_n(t)}{\tau_g(t)}\dot{\bigg)}.
	\end{equation}
	The general solution becomes
	\begin{equation}\label{eq3.k}
	\lambda_3(t)= \bigg[c_2-\int{f(t)\sin{t}}dt\bigg]\cos{t}+\bigg[c_3+\int{f(t)\cos{t}}dt\bigg]\sin{t}
	\end{equation}
	where $c_2, c_3$ are arbitrary real constants and $f(t)=\frac{\lambda_1(t)\kappa_g(t)}{\tau_g(t)}-\bigg(\frac{\lambda_1(t)\kappa_n(t)}{\tau_g(t)}\dot{\bigg)}.$ By differentiating (\ref{eq3.k}) and plug the resulting equation into (\ref{eq2.k}), we obtain
	\begin{equation}\label{eq4}
	\lambda_2(t)=-\frac{\lambda_1(t)\kappa_n(t)}{\tau_g(t)}+\bigg[c_2-\int{f(t)\sin{t}}dt\bigg]\sin{t}-\bigg[c_3+\int{f(t)\cos{t}}dt\bigg]\cos{t}.
	\end{equation}
	As a result the equations (\ref{eq3.k}) and (\ref{eq4}) becomes
	\begin{equation}\label{eq4*}
	\begin{split}
	\lambda_2(x)=-\frac{(x+c_1)\kappa_n(x)}{\tau_g(x)}&+\bigg(c_2-\int{f(x)\tau_g(x)\sin{[t(x)]}}\,dx\bigg)\sin{[t(x)]}\\&-\bigg(c_3+\int{f(x)\tau_g(x)\cos{[t(x)]}}\,dx\bigg)\cos{[t(x)]}.
	\end{split}
	\end{equation}
	and
	\begin{equation}\label{eq3*}
	\begin{split}
	\lambda_3(x)&= \bigg(c_2-\int{f(x)\tau_g(x)\sin{[t(x)]}}\,dx\bigg)\cos{[t(x)]}\\&+\bigg(c_3+\int{f(x)\tau_g(x)\cos{[t(x)]}}\,dx\bigg)\sin{[t(x)]}
	\end{split}
	\end{equation}
	
	where $f(x)=\frac{\lambda_1(x)\kappa_g(x)}{\tau_g(x)}-\bigg(\frac{\lambda_1(x)\kappa_n(x)}{\tau_g(x)}\bigg)'\frac{1}{\tau_g(x)}$ and $t(x)=\int{\tau_g(x)}\,dx$.
	
	Substituting equations (\ref{eq1.k}), (\ref{eq4*}) and (\ref{eq3*}) to (\ref{eq0}) we obtain (\ref{eq00}). This completes the proof.
\end{proof}

%\section{Position vector of a curve with respect to the standard frame in $G^3$}

\begin{thm}\label{teo2}
	The position vector $\beta(x)$ of an arbitrary curve on a surface with respect to the standard frame in the Galilean space $G^3$ is computed from the natural representation form:
	\begin{equation}\label{eq11}
	\begin{split}
	\beta(x)=\Bigg(x,& \int\bigg[\int\Big(\kappa_g(x)S_{\tau_g}-\kappa_n(x)\int{\tau_g(x)S_{\tau_g}\, dx}\Big)\, dx\bigg]\, dx,\\& \int\bigg[\int\Big(\kappa_g(x)C_{\tau_g}-\kappa_n(x)\int{\tau_g(x)C_{\tau_g}\, dx}               \Big)\, dx\bigg]\, dx\Bigg)
	\end{split}
	\end{equation}
	where $C_{\tau_g}=\cos\big[\int{\tau_g(x)}\,dx\big]$ and $S_{\tau_g}=\sin\big[\int{\tau_g(x)}\,dx\big]$.
\end{thm}

\begin{proof}If $\beta(x)$ is a curve on a surface in Galilean space $G^3$, then the Frenet
	equations (\ref{Darboux}) are hold. It is easy to see that the following differential equation is obtained by using (\ref{Darboux}). 
	$$\bigg(\frac{1}{\tau_g(x)}\mathbf{Q'(x)}\bigg)'=-\tau_g(x)\mathbf{Q(x)}$$
	The above equation can be written in the form
	\begin{equation}\label{eqQ}
	\frac{d^2\mathbf{Q}}{dt^2}+\mathbf{Q}=0,
	\end{equation}
	where $t$ is the new variable that equals to $t =\int\tau_g(x)\,dx$.
	
	Thus, we can write $\mathbf{Q}$ as follows:
	\begin{equation}\label{eqQ1}
	\mathbf{Q}=\big(0, \sin[\theta(t)], \cos[\theta(t)]\big)
	\end{equation}
	If we  substitute (\ref{eqQ1}) into (\ref{eqQ}), and solve componentwise, we get the following two equations
%	$$(1-\dot{\theta}^2(t))\cos[\theta(t)]-\ddot{\theta}(t)\sin[\theta(t)]=0$$
%	$$(1-\dot{\theta}^2(t))\sin[\theta(t)]+\ddot{\theta}(t)\cos[\theta(t)]=0$$
%	It is easy to prove that the above equations lead to the following two equations:
	$$\dot{\theta}(t)=\pm1, \hskip1cm \ddot{\theta}(t)=0$$
	which lead to $\theta(t) = \pm t = \pm \int \tau_g(x)\,dx$. Without loss of generality, we can assume that $\theta(t)$ has a positive sign. Then we get
	\begin{equation}\label{eqQx}
	\mathbf{Q}(x)=\bigg(0, \sin[\int \tau_g(x)\, dx], \cos[\int \tau_g(x)\, dx]\bigg).
	\end{equation}
	From (\ref{Darboux}), we obtain
	\begin{equation*}
	\begin{split}
	\mathbf{n}(x)&=-\int \tau_g(x)\mathbf{Q}(x)\, dx\\
	&=-\int \tau_g(x) \bigg(0, \sin[\int \tau_g(x)\, dx], \cos[\int \tau_g(x)\, dx]\bigg)\, dx + \mathbf{c} 
	\end{split}
	\end{equation*}
	where $\mathbf{c}$ is a constant vector. Since the first component of normal vector is zero, then we can take $\mathbf{c} = (0, 0, 0)$, and then
	\begin{equation}\label{eqNx}
	\mathbf{n}(x)=\Bigg(0, \cos\bigg(\int \tau_g(x)\, dx\bigg), -\sin\bigg(\int \tau_g(x)\, dx\bigg)\Bigg)
	\end{equation}
%	\begin{equation}\label{eqNx}
%	\mathbf{n}(x)=\bigg(0, -\int \tau_g(x)\sin[\int \tau_g(x)\, dx]\, dx, -\int \tau_g(x)\cos[\int \tau_g(x)\, dx]\, dx\bigg)
%	\end{equation}
	From (\ref{eqQx}) and (\ref{Darboux}), we have
	\begin{equation}
	\begin{split}
	\mathbf{T'}(x)&=\kappa_g(x)\mathbf{Q}(x)+\kappa_n(x)\mathbf{n}(x)\\
	&=\kappa_g(x)\bigg(0, \sin[\int \tau_g(x)\, dx], \cos[\int \tau_g(x)\, dx]\bigg)\\
	&+\kappa_n(x)\bigg(0, -\int \tau_g(x)\sin[\int \tau_g(x)\, dx]\, dx, -\int \tau_g(x)\cos[\int \tau_g(x)\, dx]\, dx\bigg)\\
	&=\bigg(0, \kappa_g(x) \sin[\int \tau_g(x)\, dx]- \kappa_n(x)\int \tau_g(x)\sin[\int \tau_g(x)\, dx]\, dx\\
	&\hskip1cm, \kappa_g(x) \cos[\int \tau_g(x)\, dx]-\kappa_n(x)\int \tau_g(x)\cos[\int \tau_g(x)\, dx]\, dx\bigg).
	\end{split}
	\end{equation}
	If we let $C_{\tau_g}=\cos[\int \tau_g(x)\, dx]$ and $S_{\tau_g}=\sin[\int \tau_g(x)\, dx]$, then we have
	\begin{equation}\label{T'(x)}
	\mathbf{T'}(x)=\bigg(0, \kappa_g(x) S_{\tau_g}- \kappa_n(x)\int \tau_g(x)S_{\tau_g}\, dx, \kappa_g(x) C_{\tau_g}-\kappa_n(x)\int \tau_g(x)C_{\tau_g}\, dx\bigg).
	\end{equation}
 Taking the integral of (\ref{T'(x)}) with respect to, we get
	\begin{equation}
	\mathbf{T}(x)=\bigg(0, \int\Big(\kappa_g(x) S_{\tau_g}- \kappa_n(x)\int \tau_g(x)S_{\tau_g}\, dx\Big)\,dx, \int\Big(\kappa_g(x) C_{\tau_g}-\kappa_n(x)\int \tau_g(x)C_{\tau_g}\, dx\Big)\,dx\bigg)+\mathbf{d}
	\end{equation}
	where $\mathbf{d}$ is a constant vector. Since the first component of tangent vector is one, we can take $\mathbf{d} = (1, 0, 0)$, and then
	\begin{equation}\label{T(x)}
	\mathbf{T}(x)=\bigg(1, \int\Big(\kappa_g(x) S_{\tau_g}- \kappa_n(x)\int \tau_g(x)S_{\tau_g}\, dx\Big)\,dx, \int\Big(\kappa_g(x) C_{\tau_g}-\kappa_n(x)\int \tau_g(x)C_{\tau_g}\, dx\Big)\,dx\bigg).
	\end{equation}
	Integrating (\ref{T(x)}) with respect to $x$, we have
	\begin{equation}
	\begin{split}
	\beta(x)=\Bigg(x,& \int\bigg[\int\Big(\kappa_g(x)S_{\tau_g}-\kappa_n(x)\int{\tau_g(x)S_{\tau_g}\, dx}\Big)\, dx\bigg]\, dx,\\& \int\bigg[\int\Big(\kappa_g(x)C_{\tau_g}-\kappa_n(x)\int{\tau_g(x)C_{\tau_g}\, dx}               \Big)\, dx\bigg]\, dx\Bigg)
	\end{split}
	\end{equation}
	where $C_{\tau_g}=\cos\big[\int{\tau_g(x)}\,dx\big]$ and $S_{\tau_g}=\sin\big[\int{\tau_g(x)}\,dx\big]$
	which leads to the equation (\ref{eq11}) and the proof is complete.
\end{proof}

\section{Applications}
We begin a study of important special curves lying on surfaces. For example, geodesic, asymtotic and curvature (or principal) line. 
Let $\beta$ be regular curve on the oriented surface in $G^3$ with the curvature $\kappa$, the torsion $\tau$, the geodesic curvature $\kappa_g$, the normal curvature $\kappa_n$ and the geodesic torsion $\tau_g$. 
\begin{defn}\label{defgap}
	We can say that $\beta$ is
	\begin{equation*}
	\begin{split}
	geodesic \, curve  &\Longleftrightarrow \kappa_g\equiv 0,
	\\asymptotic \, curve & \Longleftrightarrow \kappa_n\equiv 0,
	\\line \, of \, curvature  & \Longleftrightarrow \tau_g\equiv 0.
	\end{split}
	\end{equation*}
	Also, We can say that $\beta$ is called:
	\begin{equation}\label{helsal}
	\begin{array}{ccc}
	\kappa, \tau  & \hskip 1cm& \beta\\
	\hline
	\kappa\equiv0 &\Longrightarrow &\textbf{a straight line.}\\
	\tau\equiv0 &\Longrightarrow &\textbf{a plane curve.}\\
	\kappa\equiv\textit{cons.$>$0},\tau\equiv\textit{cons.$>$0} &\Longrightarrow &\textbf{a circular helix or W-curve.}\\
	\frac{\tau}{\kappa}\equiv\textit{cons.} &\Longrightarrow &\textbf{a generalized helix.}\\
	\kappa\equiv\textit{cons.}, \tau\nequiv\textit{cons.} &\Longrightarrow &\textbf{Salkowski curve \cite{mon,sal}.}\\
	\kappa\nequiv\textit{cons.}, \tau\equiv\textit{cons.} &\Longrightarrow &\textbf{anti-Salkowski curve \cite{sal}.}\\
	\end{array}
	\end{equation}
\end{defn}

\subsection{The position vector of a family of geodesic line in the Galilean space $G^3$}

\begin{thm}\label{geo}
 The position vector $\beta_g(x)$ of a family of geodesic line in Galilean space $G^3$ is given by
\begin{equation}
	\mathbf{\beta_g}(x)=\Bigg(x, -\int\int\kappa_n(x)\int\tau_g(x)S_{\tau_g}\,dx\,dx\,dx, \\ -\int\int\kappa_n(x)\int\tau_g(x)C_{\tau_g}\,dx\,dx\,dx\Bigg).
	\end{equation}
\end{thm}
\begin{proof}
	By using $\kappa_g(x)\equiv0$ in the equation (\ref{eq11}), we obtain the above equation.
\end{proof}

\begin{cor}
	The position vector of a geodesic that is a circular helix is defined by the equation
	$$\mathbf{\beta_{g_{ch}}(x)}=\Bigg(x,-\frac{e}{c^2}\cos(cx+c_1)+e_1x^2+e_2x+e_3, \frac{e}{c^2}\sin(cx+c_1)+f_1x^2+f_2x+f_3\Bigg)$$
	where $c,c_1, e, e_1, e_2, e_3, f_1, f_2$ and $f_3$ are constants.
\end{cor}
\begin{proof}
	By using the definition (\ref{defgap}) and the equations (\ref{kt}) in (\ref{geo}), we get the above equation.
\end{proof}

\begin{cor}
	The position vector of a geodesic that is a generalized helix is defined by the equation
	$$\mathbf{\beta_{g_{gh}}(x)}=\Bigg(x,\int\int\kappa_n(x)\Big[\cos\Big(d\int\kappa_n(x)dx\Big)+d_1\Big]dxdx, -\int\int\kappa_n(x)\Big[\sin\Big(d\int\kappa_n(x)dx\Big)+d_2\Big]dxdx\Bigg)$$
	where $d, d_1$ and $d_2$ are constants.
\end{cor}
\begin{proof}
	By using the definition (\ref{defgap}) and the equations (\ref{kt}), we obtain $\tau_g(x)=d\kappa_n(x)$. By using this equation in (\ref{geo}), we get the above equation.
\end{proof}

\begin{cor}
	The position vector of a geodesic that is a Salkowski curve is defined by the equation
	$$\mathbf{\beta_{g_{s}}(x)}=\Bigg(x,m\int\int\bigg( \cos\Big(\int\tau_g(x)dx\Big)+m_1\bigg)dxdx, -m\int\int\bigg( \sin\Big(\int\tau_g(x)dx\Big)+m_2\bigg)dxdx\Bigg)$$
	where $m, m_1$ and $m_2$ are constants.
\end{cor}
\begin{proof}
	By using the definition (\ref{defgap}) and the equations (\ref{kt}), we obtain $\kappa_n(x)\equiv d(const.)$ and $\tau_g(x)\nequiv const.$. By using this equation in (\ref{geo}), we get the above equation.
\end{proof}

\begin{cor}
	The position vector of a geodesic that is a anti-Salkowski curve is defined by the equation
	$$\mathbf{\beta_{g_{as}}(x)}=\Bigg(x,\int\int\Big(\kappa_n(x) \Big[\cos(bx+b_1)+b_2\Big]\Big)dxdx, -\int\int\Big(\kappa_n(x) \Big[\sin(bx+b_1)+b_3\Big]\Big)dxdx\Bigg)$$
	where $b, b_1, b_2$ and $b_3$ are constants.
\end{cor}
\begin{proof}
	By using the definition (\ref{defgap}) and the equations (\ref{kt}), we obtain $\tau_g(x)\equiv b(const.)$ and $\kappa_n(x)\nequiv const.$. By using this equation in (\ref{geo}), we get the above equation.
\end{proof}

\subsection{The position vector of a family of asymptotic line in the Galilean space $G^3$}

\begin{thm}\label{asym} The position vector $\beta_a(x)$ of a family of asymptotic line in Galilean space $G^3$ is given by
	\begin{equation}\label{asymeq}
	\mathbf{\beta_a}(x)=\Bigg(x, \int\int\kappa_g(x) \sin\bigg(\int{\tau_g(x)}\,dx\bigg)\,dx\,dx, \int\int\kappa_g(x) \cos\bigg(\int{\tau_g(x)}\,dx\bigg)\,dx\,dx\Bigg)
	\end{equation}
\end{thm}

\begin{proof}
	By using $\kappa_n(x)\equiv0$ in the equation (\ref{eq11}), we obtain the above equation.
\end{proof}
\begin{cor}
	The position vector of a asymptotic that is a circular helix is defined by the equation
	$$\mathbf{\beta_{a_{ch}}(x)}=\Bigg(x,-\frac{e}{c^2}\sin(cx+c_1)+c_2x+c_3, -\frac{e}{c^2}\cos(cx+c_1)+c_4x+c_5\Bigg)$$
	where $c,c_1,c_2, c_3, c_4, c_5$ and $e$ are constants.
\end{cor}
\begin{proof}
	By using the definition (\ref{defgap}) and the equations (\ref{kt}) in (\ref{asymeq}), we obtain $\kappa_g(x)\equiv e, \tau_g(x)\equiv c$ where $e$ and $c$ are constants. By using this relation in (\ref{asymeq}), we get the above equation.
\end{proof}

\begin{cor}
	The position vector of a asymptotic that is a generalized helix is defined by the equation
	$$\mathbf{\beta_{a_{gh}}(x)}=\Bigg(x,-\frac{1}{k}\int \cos\Big(k\int\kappa_g(x)dx\Big)dx+k_1x+k_2, \frac{1}{k}\int \sin\Big(k\int\kappa_g(x)dx\Big)dx+k_3x+k_4\Bigg)$$
	where $k, k_1,k_2,k_3$ and $k_4$ are constants.
\end{cor}
\begin{proof}
	By using the definition (\ref{defgap}) and the equations (\ref{kt}), we obtain $\tau_g(x)=k\kappa_g(x)$ where $k$ is a constant. By using this relation in (\ref{asymeq}), we get the above equation.
\end{proof}

\begin{cor}
	The position vector of a asymptotic that is a Salkowski curve is defined by the equation
	$$\mathbf{\beta_{a_{s}}(x)}=\Bigg(x,\int\int\bigg(e \sin\Big(\int\tau_g(x)dx\Big)\bigg)dxdx, \int\int\bigg(e \cos\Big(\int\tau_g(x)dx\Big)\bigg)dxdx\Bigg)$$
	where $e$ is a constant.
\end{cor}
\begin{proof}
	By using the definition (\ref{defgap}) and the equations (\ref{kt}), we obtain $\kappa_g(x)\equiv e(const.)$ and $\tau_g(x)\nequiv const.$. By using this equation in (\ref{asymeq}), we get the above equation.
\end{proof}

\begin{cor}
	The position vector of a asymptotic that is a anti-Salkowski curve is defined by the equation
	$$\mathbf{\beta_{a_{as}}(x)}=\Bigg(x,\int\int\Big(\kappa_g(x) \sin(dx+d_1)\Big)dxdx, \int\int\Big(\kappa_g(x) \cos(dx+d_1)\Big)dxdx\Bigg)$$
	where $d$ and $d_1$ are constants.
\end{cor}
\begin{proof}
	By using the definition (\ref{defgap}) and the equations (\ref{kt}), we obtain $\tau_g(x)\equiv d(const.)$ and $\kappa_g(x)\nequiv const.$. By using this equation in (\ref{asymeq}), we get the above equation.
\end{proof}

%%%%%%%%%%%%%%%%%%%%%%%%%%%
%%%%%%%%%%%%%%%%%%%%%%%%%%%
%%%%%%CURVATURE OR PRİNCİPAL%%%%%% %%%%%%%%%%%%%%%%%%%%%%%%%%%
%%%%%%%%%%%%%%%%%%%%%%%%%%%

\subsection{The position vector of a family of line of curvature in the Galilean space $G^3$}

\begin{thm}\label{cur} The position vector $\beta_c(x)$ of a family of  line of curvature in Galilean space $G^3$ is given by
	\begin{equation}
	\mathbf{\beta_p}(x)=\Bigg(x, \int\int\Big(c_1\kappa_g(x)- c_2\kappa_n(x)\Big)\,dx\,dx, \int\int\Big(c_3\kappa_g(x)-c_4\kappa_n(x)\Big)\,dx\,dx\Bigg)
	\end{equation}
	where $c_1, c_2, c_3$ and $c_4$ are constants.
\end{thm}

\begin{proof}
	By using $\tau_g(x)\equiv0$ in the equation (\ref{Darboux}), we obtain the above equation.
\end{proof}

\begin{cor}
	The position vector of a line of curvature is a circular helix if and only if the below system of differential equations is satisfied.
	\begin{equation}\label{pcheq}
	\begin{split}
	\kappa_g(x)\kappa'_g(x)+\kappa_n(x)\kappa'_n(x) &=0,\\
	\kappa_n(x)\kappa''_g(x)-\kappa_g(x)\kappa''_n(x) &=0.
	\end{split}
	\end{equation}
\end{cor}
\begin{proof}
	By using the definition (\ref{helsal}) and the equations (\ref{kt}), we get the above equation.
\end{proof}
\textit{Special Case:}

If $\kappa_g(x)$ and $\kappa_n(x)$ are constant functions, then the equation (\ref{pcheq}) is satisfied. Therefore, the line of curvature with  $\kappa_g(x)\equiv const.$ and $\kappa_n(x)\equiv const.$ is a circular helix, and its position vector is defined by
\begin{equation}
\mathbf{\beta_{p_{ch}}}(x)=\Big(x, a_1x^2+a_2x+a_3, b_1x^2+b_2x+b_3\Big)
\end{equation}
where $a_1, a_2, a_3$ and $b_1,b_2, b_3$ are constants.
\begin{cor}
	The position vector of a line of curvature is a generalized helix if and only if the below differential equation is satisfied.
	\begin{equation}\label{pgheq}
	\kappa''_g\kappa_n^3+\kappa''_g\kappa_n\kappa_g^2-\kappa_n^2\kappa_g\kappa''_n-3\kappa_n^2\kappa'_g\kappa'_n-3\kappa_n\kappa_g{\kappa'_g}^{2}	+3\kappa_n\kappa_g{\kappa'_n}^{2}-\kappa_g^3\kappa''_n+3\kappa_g^2\kappa'_g\kappa'_n=0
	\end{equation}
\end{cor}
\begin{proof}
By using the definition (\ref{helsal}) and the equations (\ref{kt}), we get the above equation.
\end{proof}

\begin{cor}\label{pseq}
	The position vector of a line of curvature is a Salkowski curve if and only if the following equation is satisfied: 
	\begin{equation*}
	\kappa_g\kappa'_g+\kappa_n\kappa'_n=0
	\end{equation*}
\end{cor}
\begin{proof}
	By using the definition (\ref{helsal}) and the equations (\ref{kt}), we get the above equation.
\end{proof}
\begin{cor}
	The position vector of a line of curvature is a anti-Salkowski curve if and only if the following equation is satisfied:
	\begin{equation}\label{paseq}
	\kappa''_g\kappa_g^2\kappa_n+\kappa''_g\kappa_n^3-\kappa''_n\kappa_g^3-\kappa''_n\kappa_g\kappa_n^2-2{\kappa'_g}^2\kappa_g\kappa_n+2\kappa'_g\kappa'_n\kappa_g^2-2\kappa'_g\kappa'_n\kappa_n^2+2{\kappa'_n}^2\kappa_g\kappa_n=0
	\end{equation}
\end{cor}
\begin{proof}
	By using the definition (\ref{helsal}) and the equations (\ref{kt}), we get the above equation.
\end{proof}

We now consider an example for geodesic curve on surface along with their graphs.
%% - figures and tables}
\begin{example}
In \eqref{geo}, if we let $\kappa(x)=\sin x$ and $\tau(x)\equiv1$, we obtain
$$\displaystyle \alpha(x) = \Bigg(x, \frac{x-\sin(x)\cos(x)}{4}, \frac{\sin(x)^2-x^2}{4}\Bigg)$$.

A surface on which this curve lies can be taken as follows:
$$\phi(u,v)=\Bigg(u+v, \frac{u-\sin(u+v)\cos(u+v)}{4}, \frac{\sin(u+v)^2-u^2}{4}\Bigg)$$
\end{example}
\begin{figure}
\centering
\includegraphics[width=0.4\textwidth]{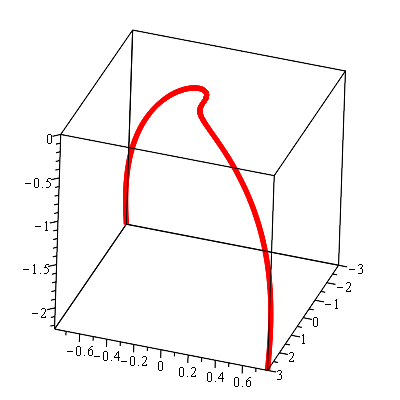}
\caption{The geodesic curve}\label{fig1:}
\end{figure}
\begin{figure}
\centering
\includegraphics[width=0.8\textwidth]{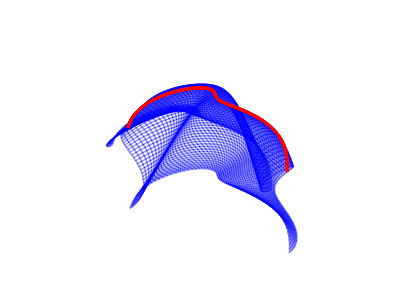}
\caption{The surface}\label{fig1:}
\end{figure}

\section{Conclusions}

This study is obtained the position vectors of all curves on a surface in $G^3$ with respect to the Darboux Frame. Firstly, the position vector of a curve on a surface in $G^3$ in terms of geodesic, normal curvature and geodesic torsion with respect to the Darboux and standard frame is investigated. As result of these, position vectors of some special curves such as geodesic, asymptotic curve, line of curvature on a surface is obtained in $G^3$.

Consequently, relations of foregoing curves with helix, Salkowski curve and anti-Salkowski curve are given(see (\ref{helsal}) ).
That is, special cases of these curves  such as: geodesics that are circular helix, genaralized helix or Salkowski, etc is given. Furthermore, the graphs of some special curves is drawn .  

In the light of these results,  we will study special smarandache curves with respect to Darboux frame in $G^3$ using these paper (arXiv 1707.03935v1). Also, we want to emphasize that the results of this study can be extended to families of surfaces that have common geodesic, asymptotic curve and line of curvature.
\ack % or \acks
% Put acknowledgements here
This study was supported financially by the Research Centre of Amasya University (Project No:  FMB-BAP16-0213).
%Last but not least,

% alteratively, bibliographies prepared with BibTeX can be included by
% means of the following commands
%\bibliographystyle{srtnumbered}
%\bibliography{mybib}

\end{document}